\documentclass[preprint,11pt]{elsarticle}
\usepackage[T1]{fontenc}
\usepackage[utf8]{inputenc}
\usepackage[english]{babel}
\usepackage{amsmath,amsthm,enumerate}
\usepackage{amssymb}
\usepackage{amscd}
\usepackage{vmargin}
\usepackage{url}
\usepackage{stmaryrd}
\usepackage{tikz,pgf}
\usepackage{subfig}
\usepackage{cancel}

\usetikzlibrary{decorations.pathmorphing}
\usetikzlibrary{shapes,snakes}
\tikzset{snake it/.style={decorate, decoration=snake}}

\newcommand\blank[1][.6em]{%
  \mbox{\kern.06em\vrule height.5ex}%
  \vbox{\hrule width#1}%
  \hbox{\vrule height.5ex}}
  
\begin{document}
\begin{frontmatter}
\title{On a combination of the 1-2-3 Conjecture and \\ the Antimagic Labelling Conjecture}

\author[nice]{Julien Bensmail}
\author[labri]{Mohammed Senhaji}
\author[dtu]{Kasper Szabo Lyngsie}

\journal{...}

\address[nice]{INRIA and Université Nice-Sophia-Antipolis \\ I3S, UMR 7271 \\ 06900 Sophia-Antipolis, France \\~}
\address[labri]{Université de Bordeaux \\ LaBRI, UMR5800 \\ F-33400 Talence, France \\~}
\address[dtu]{Department of Applied Mathematics and Computer Science \\ Technical University of Denmark \\ DK-2800 Lyngby, Denmark \\~}

\begin{abstract}
This paper is dedicated to studying the following question:
Is it always possible to injectively assign the weights $1,...,|E(G)|$ to the edges of any given graph $G$ (with no component isomorphic to $K_2$)
so that every two adjacent vertices of $G$ get distinguished by their sums of incident weights?
One may see this question as a combination of the well-known 1-2-3 Conjecture and the Antimagic Labelling Conjecture.

Throughout this paper, we exhibit evidence that this question might be true.
Benefiting from the investigations on the Antimagic Labelling Conjecture,
we first point out that several classes of graphs, such as regular graphs, indeed admit such assignments.
We then show that trees also do, answering a recent conjecture of Arumugam, Premalatha, Ba\v{c}a and Semani\v{c}ov\'a-Fe\v{n}ov\v{c}\'ikov\'a.
Towards a general answer to the question above, 
we then prove that claimed assignments can be constructed for any graph,
provided we are allowed to use some number of additional edge weights.
For some classes of sparse graphs, namely $2$-degenerate graphs and graphs with maximum average degree~$3$,
we show that only a small (constant) number of such additional weights suffices.
\end{abstract}

\begin{keyword} 
1-2-3 Conjecture, Antimagic Labelling Conjecture, equitable edge-weightings, neighbour-sum-distinguishing edge-weightings.
\end{keyword}
 
\end{frontmatter}

\newcommand\mad{{\rm mad}}
\def\mymathhyphen{{\hbox{-}}}

\newtheorem{theorem}{Theorem}[section]
\newtheorem{lemma}[theorem]{Lemma}
\newtheorem{conjecture}[theorem]{Conjecture}
\newtheorem{observation}[theorem]{Observation}
\newtheorem*{123conjecture}{1-2-3 Conjecture}
\newtheorem*{antimagic}{Antimagic Labelling Conjecture}
\newtheorem*{12conjecture}{1-2 Conjecture}
\newtheorem{claim}[theorem]{Claim}
\newtheorem{corollary}[theorem]{Corollary}
\newtheorem{proposition}[theorem]{Proposition}
\newtheorem{question}[theorem]{Question}


\section{Introduction} \label{section:introduction}

In order to present our investigations in this paper, 
as well as our motivations,
we first need to introduce a few particular graph concepts and notions.
We refer the reader to textbooks on graph theory for more details on any standard notion or terminology not introduced herein.

Given a (undirected, simple, loopless) graph $G$ and a set $W$ of weights,
by a \textit{$W$-edge-weighting} of $G$ we mean an edge-weighting with weights from $W$.
For any $k \geq 1$, a \textit{$k$-edge-weighting} is a $\{1,...,k\}$-edge-weighting.
Given an edge-weighting $w$ of $G$,
one can compute, for every vertex $v$ of $G$,
the sum $\sigma(v)$ (or $\sigma_w(v)$ when more precision is needed) of weights assigned by $w$ to the edges incident to $v$.
That is, $$\sigma_w(v) := \sum_{u \in N(v)} w(vu)$$
for every vertex $v$ of $G$.
In case we have $\sigma_w(u) \neq \sigma_w(v)$ for every edge $uv$ of $G$,
we call $w$ \textit{neighbour-sum-distinguishing}.
It can be observed that every graph with no connected component isomorphic to~$K_2$ admits neighbour-sum-distinguishing edge-weightings using sufficiently large weights.
In the context of the current investigations, when speaking of a \textit{nice graph} we mean a graph with no connected component isomorphic to~$K_2$.
For a nice graph $G$,
it hence makes sense to study the smallest $k$ such that $G$ admits a neighbour-sum-distinguishing $k$-edge-weighting.
We denote this chromatic parameter by $\chi^e_\Sigma(G)$.

Throughout this paper,
we deal with edge-weightings that are not only neighbour-sum-distinguishing
but also do not assign any edge weight more than once.
We say that such edge-weightings are \textit{edge-injective}.
Still under the assumption that $G$ is a nice graph,
we denote by $\chi^{e,1}_\Sigma(G)$ the smallest $k$ such that $G$ admits an edge-injective neighbour-sum-distinguishing $k$-edge-weighting.

\medskip

In this paper, we consider the following conjecture.
Our motivations for studying this conjecture,
as well as our evidences to suspect that it might be true,
are described below.

\begin{conjecture} \label{conjecture:main}
For every nice graph $G$, we have $\chi^{e,1}_\Sigma(G) = |E(G)|$. 
\end{conjecture}

\noindent By the edge-injectivity property, 
we note that $|E(G)|$ is a lower bound on $\chi^{e,1}_\Sigma(G)$ for every nice graph $G$.
Conjecture~\ref{conjecture:main}, in brief words,
hence asks whether, for every nice graph $G$,
we can bijectively assign weights~$1,...,|E(G)|$ to the edges of $G$
so that no two adjacent vertices of $G$ get the same value of $\sigma$.

\medskip

Conjecture~\ref{conjecture:main} is related to the well-known \textbf{1-2-3 Conjecture},
raised in 2004 by Karo\'nski, {\L}uczak and Thomason~\cite{KLT04},
which states the following.

\begin{123conjecture}
For every nice graph $G$, we have $\chi^e_\Sigma(G) \leq 3$.
\end{123conjecture}

Many aspects of the 1-2-3 Conjecture have been studied in literature.
For an overview of those considered aspects, we refer the interested reader to the wide survey~\cite{Sea12} by Seamone, which is dedicated to this topic.
Our investigations in this paper 
are mostly related to a recent equitable variant of the 1-2-3 Conjecture that was considered by Baudon, Pil\'sniak, Przyby{\l}o, Senhaji, Sopena and Wo\'zniak in~\cite{BPPSSW16}.
In this variant, the authors studied, for some families of nice graphs,
the existence of neighbour-sum-distinguishing edge-weightings being \textit{equitable},
i.e. in which any two distinct edge weights are assigned about the same number of times (being equal, or differing by~$1$).
In particular, they introduced and studied, for any given graph $G$, the chromatic parameter denoted by $\overline{\chi^e_\Sigma}(G)$
being the smallest maximal weight in an equitable neighbour-sum-distinguishing edge-weighting of $G$. 
In brief words, they proved that, at least for particular common classes of nice graphs (such as complete graphs and some bipartite graphs),
the two parameters $\chi^e_\Sigma$ and $\overline{\chi^e_\Sigma}$ are equal except for a few exceptions.

Despite their results,
the authors of~\cite{BPPSSW16} did not dare addressing a general conjecture on how should $\overline{\chi^e_\Sigma}$ behave in general,
or compared to $\chi^e_\Sigma$ for a given nice graph.
In particular, it does not seem obvious how big $\overline{\chi^e_\Sigma}$ can be, 
neither whether this parameter can be arbitrarily large.
This is one of our motivations for studying edge-injective neighbour-sum-distinguishing edge-weightings,
as an edge-injective edge-weighting is always equitable.
Thus, $\overline{\chi^e_\Sigma}(G) \leq \chi^{e,1}_\Sigma(G)$ holds for every nice graph $G$.
Hence, attacking Conjecture~\ref{conjecture:main} can be regarded as a way to get progress towards all those questions.

\medskip

Our second motivation for considering Conjecture~\ref{conjecture:main}
is that edge-injective neighbour-sum-distinguishing edge-weightings can be regarded as a weaker notion of well-known \textbf{antimagic labellings}.
Formally, using our own terminology, an \textit{antimagic labelling} $w$ of a graph $G$
is an edge-injective $|E(G)|$-edge-weighting of $G$ for which $\sigma_w$ is injective,
i.e. all vertices of $G$ get a distinct sum of incident weights by $w$.
We say that $G$ is \textit{antimagic} if it admits an antimagic labelling.
Many lines of research concerning antimagic labellings can be found in literature,
most of which are related to the following conjecture
addressed by Hartsfield and Ringel in~\cite{HR90}.

\begin{antimagic}
Every nice connected graph is antimagic.
\end{antimagic}

Despite lots of efforts (refer to the dynamic survey~\cite{Gal97} by Gallian for an in-depth summary of the vast and rich literature on this topic),
the Antimagic Labelling Conjecture is still open in general,
even for common classes of graphs such as nice trees.
Conjecture~\ref{conjecture:main}, which is clearly much weaker than the Antimagic Labelling Conjecture,
as the distinction condition here only concerns the adjacent vertices,
hence sounds as a much easier challenge to us,
in particular concerning classes of nice graphs that are not known to be antimagic.

Hence, every antimagic graph $G$ agrees with Conjecture~\ref{conjecture:main},
implying, as described earlier, that 
$$\overline{\chi^e_\Sigma}(G) \leq \chi^{e,1}_\Sigma(G) = |E(G)|$$
holds, thus providing an upper bound on $\overline{\chi^e_\Sigma}(G)$ for $G$.
This is of interest as several classes of graphs,
such as nice regular graphs and nice complete partite graphs,
are known to be antimagic, see~\cite{Gal97}.
Let us here further mention the works of B\'erci, Bern\'ath and Vizer~\cite{BBV15}, 
and of Cranston, Liang and Zhu~\cite{CLZ14},
who led to the verification of the Antimagic Labelling Conjecture for nice regular graphs,
and whose some proof techniques partly inspired some used in the current paper.
Conversely, proving that a graph $G$ verifies $\chi^{e,1}_\Sigma(G)=|E(G)|$ and agrees with Conjecture~\ref{conjecture:main}
is similar to proving that,
in some sense, $G$ is ``locally antimagic''.

\bigskip

Conjecture~\ref{conjecture:main} can essentially be considered as a combination of the 1-2-3 Conjecture and the Antimagic Labelling Conjecture,
as the notions behind it have flavours of both conjectures.
As described earlier, proving Conjecture~\ref{conjecture:main} for some classes of graphs has, to some extent,
consequences on the 1-2-3 Conjecture and the Antimagic Labelling Conjecture,
or at least on variants of these conjectures. 

Our work in this paper, is focused on both proving Conjecture~\ref{conjecture:main} for particular classes of nice graphs,
and providing upper bounds on $\chi^{e,1}_\Sigma$ for some classes of nice graphs.
This paper is organized as follows.
Tools and preliminary results we use throughout
are introduced in Section~\ref{section:prelim}.
After that, we start off by providing support to Conjecture~\ref{conjecture:main} in Section~\ref{section:support},
essentially by showing and pointing out that the conjecture holds for some classes of graphs, such as nice trees and regular graphs.
Towards Conjecture~\ref{conjecture:main}, we then provide, in Section~\ref{section:upper-bound}, 
general weaker upper bounds on~$\chi^{e,1}_\Sigma$.
These bounds are then improved for some classes of nice sparse graphs in Section~\ref{section:bound-classes}.
These classes include nice graphs with maximum average degree at most~$3$
and nice $2$-degenerate graphs.
Concluding comments are gathered in Section~\ref{section:ccl}.

\medskip

\noindent \textbf{Remark: } During the review process, we have been notified that a paper introducing the notion of ``locally antimatic graphs'',
written by Arumugam, Premalatha, Ba\v{c}a and Semani\v{c}ov\'a-Fe\v{n}ov\v{c}\'ikov\'a,
appeared online~\cite{APBS17}.
That paper and the current one consider different aspects of this notion.	
Namely, \cite{APBS17} is focused on the smallest number of colour sums by an edge-injective neighbour-sum-distinguishing $|E(G)|$-edge-weighting.
In particular, our Theorem~\ref{theorem:trees} on trees answers positively to Conjecture~2.3 in \cite{APBS17}.


\section{Preliminary remarks and results} \label{section:prelim}

In this section, we introduce several observations that will be of some use in the next sections.
Conjecture~\ref{conjecture:main} is mainly about $k$-edge-weightings;
however, to lighten some proofs,
we will rather focus on edge-weightings assigning strictly positive weights only.
The reader should keep this detail in mind.

\medskip

We start off by pointing out a few situations
in which, for a given edge $uv$ of any graph $G$,
we necessarily get $\sigma(u) \neq \sigma(v)$ by an edge-injective edge-weighting of $G$.
We omit a formal proof
as it is easily seen that these claims are true.
We note that the third item is more general, as it implies the other two.

\begin{observation} \label{observation:situations}
Let $G$ be a graph, and $w$ be an edge-injective edge-weighting of $G$.
Then, for every edge $uv$ of $G$, we have $\sigma(u) \neq \sigma(v)$ in any of the following situations:
\begin{enumerate}
	\item $d(u)=1$ and $d(v) \geq 2$;
	\item $d(u)=d(v)=2$;
	\item $d(u) \geq d(v)$ and $$\min\left\{w(uv') : v' \in N(u) \setminus \{v\}\right\} \geq \max\left\{w(vu') : u' \in N(v) \setminus \{u\}\right\}.$$
\end{enumerate}
\end{observation}

We now observe that to be able to successfully extend a partial neighbour-sum-distinguishing edge-weighting to an edge,
we need to have sufficiently distinct weights in hand for that purpose.

\begin{observation} \label{observation:extending-to-an-edge}
Let $G$ be a graph, $uv$ be an edge of $G$,
and $w$ be a neighbour-sum-distinguishing edge-weighting of $G-\{uv\}$ such that $\sigma(u) \neq \sigma(v)$.
Then $w$ can be successfully extended to $uv$,
provided we have a set $W$ of at least $d(u)+d(v)-1$ distinct strictly positive weights that can be assigned to $uv$.
\end{observation}

\begin{proof}
We note that $w$ currently must satisfy $\sigma(u) \neq \sigma(v)$,
as, otherwise, no matter what weight we assign to $uv$,
we would eventually get $\sigma_w(u)=\sigma_w(v)$.
Under that assumption,
we note that weighting $uv$ with any weight completely determines the value of both $\sigma_w(u)$ and $\sigma_w(v)$.
The value of $\sigma_w(u)$ eventually has to be different from the sums of weights incident to the $d(u)-1$ neighbours of $u$ different from $v$.
Similarly, the value of $\sigma_w(v)$ eventually has to be different from the sums of weights incident to the $d(v)-1$ neighbours of $v$ different from $u$.
The neighbours of $u$ and $v$ hence forbid us from assigning at most $d(u)+d(v)-2$ possible distinct weights to $uv$.
Now, since weighting $uv$ with distinct weights results in distinct values of $\sigma_w(u)$ and $\sigma_w(v)$,
it should be clear that we can find a correct weight for $uv$ in $W$,
provided $W$ includes at least $d(u)+d(v)-1$ distinct weights.
\end{proof}

\medskip

Throughout this paper,
several of the proofs consist 
in deleting two adjacent edges $vu_1$ and $vu_2$ from~$G$,
edge-weighting the remaining graph,
and correctly extending the weighting to $vu_1$ and $vu_2$. 
In this regard, we will often refer to the following result,
which is about the number of weights that are sufficient to weight $vu_1$ and $vu_2$.

\begin{observation} \label{observation:cut2edges}
Let $G$ be a graph having two adjacent edges $vu_1$ and $vu_2$
such that $G':=G-\{vu_1,vu_2\}$ admits a neighbour-sum-distinguishing edge-weighting $w_{G'}$.
Assume further that $d_G(u_1) \geq d_G(u_2)$,
and set $$\mu := \left(d_G(u_1)+1\right)+\max\left\{0,d_G(v)+d_G(u_2)-d_G(u_1)-1\right\}.$$
Then, assuming we have a set $W$ of at least $\mu$ distinct strictly positive weights,
we can extend $w_{G'}$ to a neighbour-sum-distinguishing edge-weighting of $G$
by assigning two distinct weights of $W$ to $vu_1$ and $vu_2$.
\end{observation}

\begin{proof}
We extend $w_{G'}$ to a neighbour-sum-distinguishing edge-weighting $w_G$ of $G$
by first assigning a weight of $W$ to $vu_1$,
and then assigning a distinct weight to $vu_2$.
We determine, in this proof, the smallest number $\mu$ of weights that $W$ should contain
so that this strategy has sufficiently many weights to be successfully applied.

We note that extending $w_{G'}$ to $vu_1$ completely determines the value of $\sigma_{w_G}(u_1)$,
while the value of $\sigma_{w_G}(v)$ is not determined until $vu_2$ is also weighted.
Hence, when first weighting $vu_1$,
we mainly have to make sure that $\sigma_{w_G}(u_1)$ does not get equal to the sum of weights incident to a neighbour of $u_1$ different from $v$.
Also, we should make sure that $\sigma_{w_{G'}}(v)+w_G(vu_1)$ does not get equal to $\sigma_{w_{G'}}(u_2)$,
as otherwise we would necessarily get $\sigma_{w_G}(v)=\sigma_{w_G}(u_2)$
no matter how we weight $vu_2$.
There are hence $d_G(u_1)$ conflicts to take into account when weighting $vu_1$.
Provided $W$ includes at least $d_G(u_1)+1$ distinct weights,
we can hence weight $vu_1$ correctly, i.e. so that we avoid all conflicts mentioned above, with one weight from $W$,
since assigning different weights to $vu_1$ alters $\sigma_{w_G}(u_1)$ in distinct ways.

Now assume $vu_1$ has been weighted with the additional property that $\sigma_{w_{G'}}(v)+w_G(vu_1) \neq \sigma_{w_{G'}}(u_2)$.
Since that property holds,
Observation~\ref{observation:extending-to-an-edge} tells us that we can correctly extend $w_{G'}$ to $vu_2$
provided $W \setminus \{w_G(vu_1)\}$ includes at least $d_G(v)+d_G(u_2)-1$ distinct weights.
We hence need $W \setminus \{w_G(vu_1)\}$ to include that many distinct weights.

As explained above, $W$ necessarily includes at least $d_G(u_1)$ weights that were not assigned to $vu_1$.
Hence, to make sure, after weighting $vu_1$, that $W$ still includes at least $d_G(v)+d_G(u_2)-1$ distinct weights,
we need $W$ to include at least
$$\left( d_G(v)+d_G(u_2)-1\right) - d_G(u_1)$$
other weights.
This quantity can be negative, as, notably, $vu_1$ may need a lot of weights to be weighted.
Hence $$\mu = \left(d_G(u_1)+1\right)+\max\left\{0,d_G(v)+d_G(u_2)-d_G(u_1)-1\right\},$$
as claimed,
and, under the assumption that $W$ has size $\mu$,
we can achieve the extension of $w_{G'}$ to $G$
as described earlier.
\end{proof}

\medskip

In our proofs, we will also use the fact that, in some situations, 
pendant edges can easily be weighted
assuming we are provided enough distinct weights.

\begin{observation} \label{observation:remove-pendant}
Let $G$ be a graph having a pendant edge $vu$,
where $u$ is the degree-$1$ vertex,
such that $G':=G-\{uv\}$ admits a neighbour-sum-distinguishing edge-weighting $w_{G'}$.
Then, assuming we have a set $W$ of at least $d_{G}(v)$ distinct strictly positive weights,
we can extend $w_{G'}$ to a neighbour-sum-distinguishing edge-weighting of $G$
by assigning a weight of $W$ to $vu$.
\end{observation}

\begin{proof}
Following Observation~\ref{observation:situations},
when extending $w_{G'}$ to $vu$,
we do not have to care whether $\sigma(u)$ gets equal to $\sigma(v)$.
We thus just have to make sure that $\sigma(v)$ does not get equal to the sum of weights incident to one of its neighbours in $G'$.
Recall that assigning distinct weights to $vu$ results in different sums as $\sigma(v)$.
Therefore, since $v$ has $d_G(v)-1$ neighbours in $G'$ while $W$ has size at least $d_G(v)$,
there is necessarily a weight in $W$ that can be assigned to $vu$
such that no conflict is created.
An extension of $w_{G'}$ to $G$ hence exists.
\end{proof}


\section{Classes of graphs agreeing with Conjecture~\ref{conjecture:main}} \label{section:support}

As mentioned in Section~\ref{section:introduction},
we directly benefit, in the context of Conjecture~\ref{conjecture:main}, from the investigations on antimagic labellings,
as antimagic graphs verify Conjecture~\ref{conjecture:main}.
Following the survey~\cite{Gal97} by Gallian,
the following classes of nice graphs hence agree with Conjecture~\ref{conjecture:main}.

\begin{theorem} \label{theorem:antimagic-classes}
The classes of known antimagic graphs notably include:
\begin{itemize}
	\item nice paths (Hartsfield, Ringel~\cite{HR90}),
	\item wheels (Hartsfield, Ringel~\cite{HR90}),
	\item nice regular graphs (B\'erci, Bern\'ath, Vizer~\cite{BBV15}),
	\item nice complete partite graphs (Alon, Kaplan, Lev, Roditty, Yuster~\cite{AKLRY04}).
\end{itemize}
Consequently, every of these graphs $G$ verifies $\chi^{e,1}_\Sigma(G) = |E(G)|$.
\end{theorem}

When it comes to nice graphs with maximum degree~$2$,
it is easily seen, as we are assigning strictly positive weights only, 
that any edge-injective edge-weighting is neighbour-sum-distinguishing.
Disjoint unions of nice paths and cycles hence agree with Conjecture~\ref{conjecture:main}.

\begin{observation} \label{observation:small-degree}
Let $G$ be a nice graph with $\Delta(G)=2$.
Then any edge-injective edge-weighting of $G$ is neighbour-sum-distinguishing.
\end{observation}

One of the main lines of research concerning antimagic labellings
is to determine whether nice trees are all antimagic.
In the following result, we prove that this question can be answered positively
when relaxed to edge-injective neighbour-sum-distinguishing edge-weightings.
We actually prove a stronger statement that will be useful in the next sections.

\begin{theorem}\label{theorem:trees}
Let $F$ be a nice forest.
Then, for every set $W$ of $|E(F)|$ distinct strictly positive weights,
there exists an edge-injective neighbour-sum-distinguishing $W$-edge-weighting of $F$.
In particular, we have $\chi_{\Sigma}^{e,1}(F)=|E(F)|$.
\end{theorem}

\begin{proof}
If $\Delta(F)=2$, then the result follows from Observation~\ref{observation:small-degree}.
So the claim holds whenever $F$ has size~$2$.
Assume now that the claim is false, and let $F$ be a counterexample that is minimum in terms of $n_F+m_F$,
where $n_F:=|V(F)|$ and $m_F:=|E(F)|$.
By the remark above, we have $m_F \geq 3$.
Let $W:=\{\alpha_1, ..., \alpha_{m_F}\}$ be a set of distinct strictly positive integers such that $F$ does not admit an edge-injective neighbour-sum-distinguishing $W$-edge-weighting. 
Free to relabel the weights in $W$,
we may suppose that $\alpha_1 < ... < \alpha_{m_F}$.
Due to the minimality of $F$,
we may assume that $F$ is a tree (as otherwise we could invoke the induction hypothesis).
Furthermore, we may assume that $F$ has maximum degree at least~$3$ (at otherwise Observation~\ref{observation:small-degree} would apply).

We now successively show that $F$, because it is a counterexample to the claim, cannot contain certain structures,
until we reach the point where $F$ is shown to not exist at all, a contradiction.
In particular, we focus on the length of the pendant paths of $F$,
where a \textit{pendant path} of $F$ is a maximal path $v_k...v_1$, where $k \geq 2$, such that $d(v_k) \geq 3$, 
$d(v_{k-1})=...=d(v_2)=2$, and $d(v_1)=1$.
In the case where $k=2$, we note that the pendant path is a pendant edge,
in which case $v_k=v_2$ and we have $d(v_2) \geq 3$.
Since $\Delta(F) \geq 3$,
there are at least three pendant paths in $F$.

We start off by showing that the pendant paths of~$F$ all have length at most~$2$.

\begin{claim} \label{claim:3chain}
Every pendant path of $F$ has length at most~$2$.
\end{claim}

\begin{proof}
Assume $F$ has a pendant path $P:=v_k...v_1$ with $k \geq 4$,
where $d(v_k) \geq 3$.
In this case, let $F':=F-\{v_{k-1}v_{k-2},...,v_2v_1\}$ be the tree obtained by removing, from $F$, all edges of $P$ but the one incident to $v_k$.
Clearly, $F'$ is nice and, due to the minimality of $F$,
there exists an edge-injective neighbour-sum-distinguishing $\{\alpha_1, ... , \alpha_{m_{F'}}\}$-edge-weighting $w_{F'}$ of $F'$, 
where $m_{F'}:=|E(F')|$.
To prove that the claim holds,
we have to prove that we can extend $w_{F'}$ to the edges $v_{k-1}v_{k-2},...,v_2v_1$, hence to $F$,
using weights $\alpha_{m_{F'}+1},...,\alpha_{m_F}$,
so that we get an edge-injective neighbour-sum-distinguishing $W$-edge-weighting of $F$,
a contradiction.

Due to the length of $P$, we have $|\{\alpha_{m_{F'}+1},... ,\alpha_{m_F}\}| \geq 2$.
When weighting the edges $v_{k-1}v_{k-2},...,v_2v_1$,
we note that we cannot create any sum conflicts involving any two consecutive vertices in $\{v_1,...,v_{k-1}\}$.
That is, the incident sums of any two of these vertices can never get equal.
This is according to Observation~\ref{observation:situations} since we are assigning weights injectively.
Hence, when extending $w_{F'}$,
we just have to make sure that $\sigma(v_{k-1})$ gets different from $\sigma(v_k)$,
which is possible as we have at least two distinct edge weights to work with.
So we can assign a weight to $v_{k-1}v_{k-2}$ which avoids that conflict,
and then arbitrarily extend the weighting to the edges $v_{k-2}v_{k-3},...,v_2v_1$.
This yields an edge-injective neighbour-sum-distinguishing $W$-edge-weighting of $F$.
\end{proof}

\medskip

Now designate a vertex $r$ with degree at least~$3$ of $F$ as being the \textit{root} of $F$.
This naturally defines, in the usual way, an orientation of $F$ from its root to its leaves.
For every vertex $v$ of $F$, the \textit{father} $f(v)$ of $v$ is the neighbour of $v$ which is the closest from~$r$ (if any).
Conversely, the \textit{descendants} of $v$ are all vertices, different from $v$, in the subtree of $F$ rooted at $v$ (if any).
We note that $r$ has no father, while the leaves of $F$ have no descendants.
The descendants of $v$ adjacent to $v$ (if any) are called its \textit{children}.

A \textit{multifather} $v$ of $F$ is a vertex with degree at least~$3$,
i.e. having at least two children.
In case all descendants of $v$ have degree at most~$2$,
we call $v$ a \textit{last multifather} of $F$.
In other words, a last multifather is a vertex with at least two pendant paths attached.
Since $\Delta(F) \geq 3$, there are last multifathers in $F$.

To further study the structure of $F$,
we now prove properties of its last multifathers,
still under the assumption that $F$ is rooted at a vertex $r$ with degree at least~$3$.

\begin{claim} \label{claim:root}
Vertex $r$ is not a last multifather.
\end{claim}

\begin{proof}
Assume the contrary.
Then $r$ is the only vertex with degree at least~$3$ of $F$.
In other words, $F$ is a subdivided star.
Then it should be clear that assigning the weights $\alpha_{m_{F}}, \alpha_{m_{F}-1}, ... ,\alpha_1 $, following this order, 
to the edges of $F$ as they are encountered during a breadth-first search algorithm performed from $r$
results in a neighbour-sum-distinguishing edge-weighting of $F$.
To be convinced of this statement,
one can e.g. refer to Observation~\ref{observation:situations}.
\end{proof}

\medskip

Due to Claim~\ref{claim:root},
we may assume that the root $r$ of $F$ is not a last multifather.
Then all last multifathers of $F$ (there are some)
are different from $r$,
and hence have a father.
We now refine Claim~\ref{claim:3chain} to the following.

\begin{claim} \label{claim:2chain}
Every pendant path attached to a last multifather of $F$ has length~$1$.
\end{claim}

\begin{proof}
Let $v \neq r$ be a last multifather of $F$,
and assume $v$ is incident to pendant paths with length~$2$.
We recall that all pendant paths attached to $v$ have length at most~$2$ (Claim~\ref{claim:3chain}),
and, since $v$ is a last multifather,
it is incident to at least two pendant paths.
Let $F'$ be the tree obtained from $F$ by removing all pendant paths attached to $v$.
Because $m_{F'}:=|E(F')|$ is smaller than $m_F$,
there exists an edge-injective neighbour-sum-distinguishing $\{\alpha_1, ... , \alpha_{m_{F'}}\}$-edge-weighting $w_{F'}$ of $F'$.
For contradiction,
we prove below that $w_{F'}$ can be extended correctly to the pendant paths attached to $v$
using the weights among $\{\alpha_{m_{F'}+1},... ,\alpha_{m_F}\}$ injectively.

Let $b \geq 1$ be the number of pendant paths of length~$2$ attached to $v$ in $F$,
and let $vx_1y_1,...,vx_by_b$ denote those paths (so that the $x_i$'s have degree~$2$ in $F$,
while the $y_i$'s have degree~$1$).
Vertex $v$ is also adjacent to $c \geq 0$ leaves $x_{b+1},...,x_{b+c}$, which are, in some sense, pendant paths of length~$1$.
Since $v$ is a multifather, we recall that $b+c=d_F(v)-1 \geq 2$.

We extend $w_{F'}$ to the edges of the pendant paths attached to $v$ in the following way.
First, we injectively arbitrarily assign the $d_F(v)-1$ weights in $\{\alpha_{m_F-d_F(v)+2},... , \alpha_{m_F}\}$ to the edges $vx_2,...,vx_{b+c}$.
After that, we assign to the edge $vx_1$ one of the weights $\alpha_{m_F-d_F(v)+1}$ or $\alpha_{m_F-d_F(v)}$
chosen so that $\sigma_{w_F}(v)$ is different from the sum of weights incident to $f(v)$, the father of $v$, by $w_{F'}$.
We then assign to $x_1y_1$ the one weight of $\alpha_{m_F-d_F(v)+1}$ or $\alpha_{m_F-d_F(v)}$ not assigned to $vx_1$.
We note that 
no matter how we complete the extension of $w_{F'}$,
eventually $\sigma_{w_F}(v)$ will be strictly bigger than $\sigma_{w_F}(x_1)$.

We finish the extension of $w_{F'}$ to $F$
by arbitrarily injectively assigning the remaining non-used smaller weights to the edges $x_1y_1,...,x_by_b$.
Because all the $x_i$'s have degree~$2$ and the $y_i$'s have degree~$1$,
no conflict may arise between those vertices (Observation~\ref{observation:situations}).
Furthermore, since the degree of $v$ is larger than the degree of the $x_i$'s,
and the weights assigned to the $vx_i$'s are bigger than the weights assigned to the $x_iy_i$'s (with possibly the exception of $vx_1$ and $x_1y_1$, which we have discussed above),
it should be clear that no conflict may arise between $v$ and the $x_i$'s (again according to Observation~\ref{observation:situations}).
So we eventually get an edge-injective neighbour-sum-distinguishing $W$-edge-weighting of $F$, a contradiction.
\end{proof}

\medskip

We finally study last multifathers of $F$ being at maximum distance from $r$.
We call these vertices the \textit{deepest last multifathers} of $F$.
From now on, we focus on a fixed deepest last multifather $v^*$ of $F$, which we choose arbitrarily.
In the upcoming proof, 
for any vertex $v$ of $F$,
we denote by $F_v$ the subtree of $F$ rooted at $v$.
Recall that all children of a last multifather are leaves (Claim~\ref{claim:2chain}). 

\begin{claim}\label{claim:grandfatherlevel}
Every last multifather $v$ of $F_{f(v^*)}$ is a child of $f(v^*)$.
In other words, $v$ is a deepest last multifather of $F$.
\end{claim}

\begin{proof}
The claim follows from the fact that if there exists a descendant $v \neq v^*$ of $f(v^*)$ being at distance at least~$2$ from $f(v^*)$,
then $v$ would, in $F$, be at greater distance from $r$ than $v^*$ is.
This would contradict the fact that $v^*$ is a deepest last multifather.
\end{proof}

\medskip

\begin{figure}[t]
	\center 
	\begin{tikzpicture}[inner sep=0.7mm]
		\draw [rounded corners,color=gray,fill=gray!20] (0,0)--(2,-4)--(-2,-4)--cycle;	
		\node at (0,-2.5) {$F$};
	
	    	\node[draw,circle,line width=1pt, fill=black] (r) at (0,0)[label=above:$r$]{};	
	    	\node[draw,circle,line width=1pt, fill=black] (fv) at (0,-4)[label=above:$f(v^*)$]{};	
	    	
	    	\node[draw,circle,line width=1pt, fill=black] (v1) at (-1.5,-5)[label=left:\scriptsize type-$1$]{};	
	    	\node[draw,circle,line width=1pt, fill=black] (v2) at (0,-5)[label=left:\scriptsize type-$2$]{};	
	    	\node[draw,circle,line width=1pt, fill=black] (v22) at (0,-6){};
	    	
	    	\node[draw,circle,line width=1pt, fill=black] (v3) at (1.5,-5)[label=right:\scriptsize type-$3$, label=above:$v^*$]{};	
	    	\node[draw,circle,line width=1pt, fill=black] (v31) at (3,-6){};
	    	\node[draw,circle,line width=1pt, fill=black] (v32) at (2.5,-6){};
	    	\node[draw,circle,line width=1pt, fill=black] (v33) at (2,-6){};
	    	\node[draw,circle,line width=1pt, fill=black] (v34) at (1.5,-6){};
	    	\node[draw,circle,line width=1pt, fill=black] (v332) at (2,-7){};
	    	\node[draw,circle,line width=1pt, fill=black] (v342) at (1.5,-7){};

	    	\draw[-,line width=1pt] (fv) -- (v1);
	    	\draw[-,line width=1pt] (fv) -- (v2) -- (v22);
	    	\draw[-,line width=1pt] (fv) -- (v3);
	    	\draw[-,line width=1pt] (v3) -- (v33);	\draw[-,line width=1pt] (v3) -- (v34);	\draw[-,dotted,line width=1pt] (v33) -- (v34);	\draw[-,dotted,line width=1pt] (v332) -- (v342);
	    	\draw[-,line width=1pt] (v3) -- (v31);	\draw[-,line width=1pt] (v3) -- (v32);	\draw[-,dotted,line width=1pt] (v31) -- (v32);
	    	\draw[-,line width=1pt] (v33) -- (v332);		\draw[-,line width=1pt] (v34) -- (v342);
	\end{tikzpicture}

\caption{Illustration of the three child types mentioned in the proof of Theorem~\ref{theorem:trees}.}
\label{figure:tree}
\end{figure}
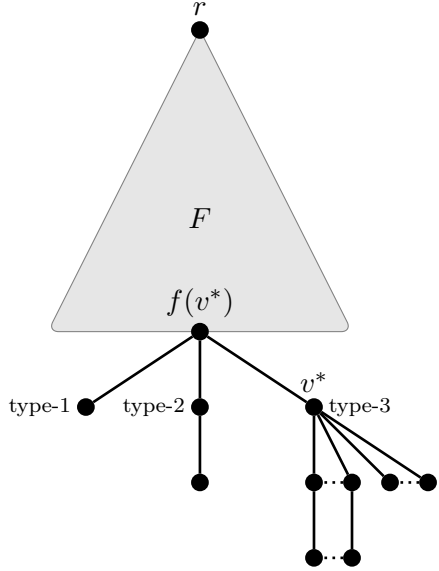

Recall that $f(v^*)$ cannot be incident, in $F$, to a pendant path with length at least~$3$ (Claim~\ref{claim:3chain}).
Hence, every child of $f(v^*)$ is either a leaf (type-1), a degree-$2$ vertex adjacent to a leaf (type-2, i.e. the inner vertex of a pendant path with length~$2$),
or a deepest last multifather (type-3).
See Figure~\ref{figure:tree} for an illustration.
Furthermore, we know that $f(v^*)$ is adjacent to at least one type-3 vertex, which is $v^*$.
In the following proof,
we show that $v^*$ is actually the only child of $f(v^*)$ in $F$. 

\begin{claim}\label{claim:grandFatherDegree}
Vertex $v^*$ is the only child of $f(v^*)$ in $F$.
\end{claim}

\begin{proof}
Suppose the claim is false, and let $v \neq v^*$ be another child of $f(v^*)$. Let $x_1$ and $x_2$ be two leaves adjacent to $v^*$, which exist since $v^*$ is a last multifather,
and all pendant paths attached to $v^*$ have length~$1$ (Claim~\ref{claim:2chain}).

\medskip

Assume first that $v$ is type-2 or type-3, or, in other words, that $d_F(v) \geq 2$. 
In that case, $v$ is adjacent to at least one leaf, say~$y$.
We here consider $F':=F-\{vy,v^*x_1,v^*x_2\}$.
Note that $F'$ remains nice and has fewer edges than $F$.
Due to the minimality of $F$, 
there hence exists an edge-injective neighbour-sum-distinguishing $\{\alpha_{1},...,\alpha_{m_{F'}}\}$-edge-weighting $w_{F'}$ of $F'$,  where $m_{F'}:=|E(F')|$.
We show below that $w_{F'}$ can be extended to the three removed edges
with injectively using the three edge weights $\alpha_{m_F-2},\alpha_{m_F-1},\alpha_{m_F}$,
yielding an edge-injective neighbour-sum-distinguishing $W$-edge-weighting $w_{F}$ of $F$, 
a contradiction.

We first assign a weight to $v^*x_1$
based on the conflicts that may happen when weighting $vy$.
When assigning any of the three weights to $vy$,
the only problem which may occur,
recall Observation~\ref{observation:situations},
is that $\sigma_{w_F}(v)$ gets equal to $\sigma_{w_{F'}}(f(v^*))$.
If assigning one of the three weights $\alpha_{m_F-2},\alpha_{m_F-1},\alpha_{m_F}$ to $vy$ indeed results in that conflict,
we assign that weight to $v^*x_1$.
Otherwise, we assign any of the three weights to $v^*x_1$.
In any case, no conflict may arise as $\sigma_{w_F}(v^*)$ is still not determined.

We are now left with two weights, which we must assign to $v^*x_2$ and $vy$.
Due to the choice of the weight assigned to $v^*x_1$,
we note that no problem may occur when weighting $vy$.
Hence, we just have to weight $v^*x_2$ correctly and assign the remaining weight to $vy$.
When weighting $v^*x_2$, the only problem which may occur,
according to Observation~\ref{observation:situations},
is that $\sigma_{w_F}(v^*)$ gets equal to $\sigma_{w_{F'}}(f(v^*))$.
But, since we have two distinct weights to work with,
one of them can be assigned to $v^*x_2$ so that this conflict is avoided.
Thus we can weight $v^*x_2$ correctly and eventually weight $vy$ with the remaining weight,
resulting in the claimed $w_F$.

\medskip

We may now assume that all children, including $v$, of $f(v^*)$ different from $v^*$ are type-1, i.e. leaves.
The contradiction can then be obtained quite similarly as in the previous case
but with setting $F':=F-\{f(v^*)v,v^*x_1,v^*x_2\}$.
When weighting $f(v^*)v$,
we have to make sure, if $f(v^*) \neq r$, that $\sigma_{w_F}(f(v^*))$ does not get equal to $\sigma_{w_{F'}}(f(f(v^*)))$.
Note that if $f(v^*)=r$, then the situation is actually easier
as there is one less conflict to consider.
If one of the three available weights $\alpha_{m_F-2},\alpha_{m_F-1},\alpha_{m_F}$, when assigned to $f(v^*)v$,
yields a conflict involving $f(v^*)$ and $f(f(v^*))$,
then we assign that weight to $v^*x_1$.
Otherwise, we assign any weight to $v^*x_1$.
This ensures that, when assigning any of the two remaining weights to $f(v^*)v$,
no conflict may involve $f(v^*)$ and $f(f(v^*))$.
We finally arbitrarily assign the two remaining weights to $v^*x_2$ and $f(v^*)v$.
If this results in a neighbour-sum-distinguishing edge-weighting $w_F$ of $F$,
then we are done.
Otherwise, it means that $\sigma_{w_F}(v^*)=\sigma_{w_F}(f(v^*))$.
In that case, note that, because all assigned edge weights are distinct,
when swapping the values assigned to $v^*x_2$ and $f(v^*)v$ by $w_F$
that conflict cannot remain.
Furthermore, according to the remarks above,
we still do not create any sum conflict involving $f(v^*)$ and $f(f(v^*))$.
After the swapping operation $w_F$ hence gets neighbour-sum-distinguishing.
\end{proof}

\medskip

We are now ready to finish off the proof
by showing that, under all information we have obtained, 
$F$ actually admits an edge-injective neighbour-sum-distinguishing $W$-edge-weighting, a contradiction.

From Claim~\ref{claim:grandFatherDegree},
we get that $d_F(f(v^*))=2$, as $v^*$ is not the root of $F$,
so $f(f(v^*))$ exists.
Let $x_1,...,x_k$ be the $k \geq 2$ leaves attached to $v^*$ in $F$, which exist since $v^*$ is a type-3 vertex.
Now consider the tree $F':=F-\{v^*x_1,...,v^*x_k\}$ with size $m_{F'}:=|E(F')|$.
Due to the minimality of $F$, there exists an edge-injective neighbour-sum-distinguishing $\{\alpha_{1},...,\alpha_{m_{F'}}\}$-edge-weighting $w_{F'}$ of $F'$.
We extend $w_{F'}$ to the $k$ removed edges so that an edge-injective neighbour-sum-distinguishing $W$-edge-weighting $w_F$ of $F$ is obtained, a contradiction.
To that aim, we arbitrarily injectively assign the weights $\alpha_{m_F-k+1},...,\alpha_{m_F}$ to the pendant edges $v^*v_1,...,v^*v_k$ attached to $v^*$.
Recall that we cannot get sum conflicts involving $v^*$ and the $v_i$'s according to Observation~\ref{observation:situations}.
Furthermore, we have $d_F(v^*) \geq 3$ while $d_F(f(v^*))=2$ (Claim~\ref{claim:grandFatherDegree}),
and we have used the $k$ biggest weights of $W$ to weight the edges incident to $v^*$.
From this and Observation~\ref{observation:situations}, we get that, necessarily, $\sigma_{w_F}(v^*)>\sigma_{w_{F'}}(f(v^*))$.
So $w_F$ is neighbour-sum-distinguishing.
\end{proof}


\section{General upper bounds} \label{section:upper-bound}

Towards Conjecture~\ref{conjecture:main}, we start off by exhibiting, for any nice graph $G$,
an upper bound on $\chi^{e,1}_\Sigma(G)$ of the form $k \cdot |E(G)|$,
where $k$ is a fixed constant.

\medskip

It turns out, first, that some results towards the 1-2-3 Conjecture
can be extended to the edge-injective context,
hence yielding bounds to our context.
This is in particular the case of the weighting algorithm by Kalkowski, Karo\'nski and Pfender from~\cite{KKP10},
which was designed to prove that $\chi^e_\Sigma(G) \leq 5$ holds for every nice graph $G$.
In very brief words, this algorithm initially assigns the list of weights $\{1,2,3,4,5\}$ to every edge of $G$,
which contains the possible weights that any edge can be assigned at any moment of the algorithm.
The algorithm then linearly processes the vertices of $G$
with possibly adjusting some incident edges weights (but staying in the list $\{1,2,3,4,5\}$)
so that sum conflicts are avoided around any vertex considered during the course.

It is easy to check that this algorithm also works
under the assumption that every edge of $G$ is assigned a (possibly unique) list of five allowed consecutive weights $\{\alpha-2,\alpha-1,\alpha,\alpha+1,\alpha+2\}$.
In particular, when applied with non-intersecting such lists assigned to the edges,
the algorithm yields an edge-injective neighbour-sum-distinguishing edge-weighting,
as every edge weight can be assigned to at most one edge. 
So, applying the algorithm on a nice graph $G$ with edges $e_0,...,e_{m-1}$ where each edge $e_i$ is assigned the list $\{5i+1,5i+2,5i+3,5i+4,5i+5\}$
results in an edge-injective neighbour-sum-distinguishing $(5 \cdot |E(G)|)$-edge-weighting of $G$.
From this, we get that $\chi^{e,1}_\Sigma(G) \leq 5 \cdot |E(G)|$ holds for every nice graph $G$.

\medskip

The $5 \cdot |E(G)|$ bound on $\chi^{e,1}_\Sigma(G)$ above can actually be improved down to $2 \cdot |E(G)|$
by means of a careful inductive proof scheme, which we describe in the following proof.
We actually prove (here and further) a stronger statement to get rid of the non-connected cases.

\begin{theorem} \label{theorem:2m}
Let $G$ be a nice graph.
Then, for every set $W$ of $2 \cdot |E(G)|$ distinct strictly positive weights,
there exists an edge-injective neighbour-sum-distinguishing $W$-edge-weighting of $G$.
In particular, we have $\chi_{\Sigma}^{e,1}(G) \leq 2 \cdot |E(G)|$.
\end{theorem}

\begin{proof}
The proof is by induction on $n_G+m_G$, where $n_G:=|V(G)|$ and $m_G:=|E(G)|$.
As it can easily be checked that the claim is true for small values of $n_G+m_G$,
we proceed to the induction step.
Consider hence a value of $n_G+m_G$
such that the claim is true for smaller values of this sum.

\medskip

We may assume that $G$ is connected,
as otherwise induction could be invoked on the different connected components of $G$.
Set $\Delta:=\Delta(G)$.
Since we may assume that $m_G \geq 4$ and $G$ is nice, we clearly have $\Delta \geq 2$.
We may even assume that $\Delta \geq 3$,
as otherwise $G$ would admit an edge-injective neighbour-sum-distinguishing $W$-edge-weighting
according to Observation~\ref{observation:small-degree}.
Consider any vertex $v^*$ of $G$ verifying $d_G(v^*)=\Delta$
and denote by $u_1, ... ,u_\Delta$ the neighbours of $v^*$ in $G$.

Set $G':=G-v^*$.
Note that $G'$ may include connected components isomorphic to $K_2$,
and thus be not nice.
In this context, we say that a component of $G'$ is \textit{empty} if it has no edge,
\textit{bad} if it is isomorphic to~$K_2$, and \textit{good} otherwise.
Basically, a bad component of $G'$ is an edge to which $v^*$ is joined in $G$:
either $v^*$ is adjacent to the two ends of that edge, 
or $v^*$ is adjacent to only one of the two ends.

If $G'$ does not have good components,
then $G$ is a connected graph whose only vertex with degree at least~$3$ is $v^*$
such that $G'$ consists of isolated vertices and isolated edges only.
In particular, all vertices of $G$ but $v^*$ have degree at most~$2$,
and every degree-$2$ vertex $u_i$ adjacent to $v^*$ is either adjacent to another degree-$2$ vertex $u_j$ adjacent to $v^*$,
or adjacent to a degree-$1$ vertex.
In such a situation,
assuming $W:=\{\alpha_1,...,\alpha_{2m_G}\}$ where $\alpha_1 < ... < \alpha_{2m_G}$,
it can easily be seen that assigning decreasing weights~$\alpha_{2m_G},...,\alpha_1$, following this order, 
to the edges of $G$ as they are encountered while performing a breadth-first search algorithm from $v^*$,
results in an edge-injective neighbour-sum-distinguishing edge-weighting of $G$.
This notably follows as a consequence of Observation~\ref{observation:situations}.

\medskip

Hence we may assume that $G'$ has good connected components $C_1, C_2, ... $ .
Let $H$ denote the union of the $C_i$'s, and set $m_{H}:=|E(H)|$. 
Since the $C_i$'s are nice, so is $H$.
Furthermore, we have that $m_{H} < m_G$.
According to the induction hypothesis, there hence exists an edge-injective neighbour-sum-distinguishing $\{\alpha_1,...,\alpha_{2m_H}\}$-edge-weighting $w_{H}$ of $H$.
In order to get an edge-injective neighbour-sum-distinguishing $W$-edge-weighting $w_G$ of $G$, we eventually need to extend $w_{H}$ to the remaining edges of $G$,
i.e. to the $v^*u_i$'s and the edges of the bad components of $G'$.

To that aim, we restrict ourselves to injectively using weights among $\{\alpha_{2m_{H}+1}, ... ,\alpha_{2m_G}\}$,
i.e. we do not use non-used weights among $\{\alpha_1,...,\alpha_{2m_H}\}$.
Let $u_1, ... , u_k$ denote the neighbours of $v^*$ belonging to good components of $G$. 
We start by injectively assigning weights to the edges $v^* u_1, ... , v^* u_k$ using $\Delta +k$ of the weights in $\{\alpha_{2m_G - (\Delta + k)+1}, ... , \alpha_{2m_G}\}$, without raising any sum conflict.
This is possible for every considered edge $v^{*} u_i$, since each $u_i$ has degree at most $\Delta - 1$ in $H$ and we have at least $\Delta +k -(i-1) \ge \Delta +1$ different available weights.

We are now left with weighting the edges of $G$ belonging to the bad components, or being incident to the bad components (i.e. being incident to $v^*$).
Assume there are $m'$ of them. Then we have $m_G=m_H+k+m'$, and, 
since $k+m'\ge \Delta$, we have 
$$2m_G - (\Delta + k) - 2m_H= k+2m'- \Delta \ge m' .$$ 
The set $\{\alpha_{2m_H+1}, ... , \alpha_{2m_G - (\Delta + k)}\}$ hence contains sufficiently many weights for weighting all of the $m'$ remaining edges. 
To that aim, we assign the weights $\alpha_{2m_G - (\Delta + k)}, ..., \alpha_{2m_H+1}$, following this order (i.e. in decreasing order of magnitude),  
to these $m'$ remaining edges as they are encountered during a breadth-first search algorithm performed from $v^*$.

It can easily be checked that, by the weighting scheme described above, 
the weights on the edges incident to $v^*$ are greater than all the weights on the edges incident to the neighbours of $v^*$.
Hence, by Observation~\ref{observation:situations}, vertex $v^*$ is distinguished from all its neighbours. 
By similar arguments, it can be checked that no sum conflicts can involve vertices of $G-H$, thus that the resulting edge-injective edge-weighting is neighbour-sum-distinguishing.
\end{proof}

\medskip


We now provide a second upper bound on $\chi^{e,1}_\Sigma(G)$ of the form $|E(G)|+k$
for every nice graph $G$.
Here, our $k$ is a small linear function of $\Delta(G)$,
making the bound 1) mostly interesting in the context of nice graphs with bounded maximum degree,
and 2) generally better than the bound in Theorem~\ref{theorem:2m}
(except in some cases to be discussed later).
The proof scheme we employ here
is different from the one used to prove Theorem~\ref{theorem:2m}.

\begin{theorem} \label{theorem:m+2D}
Let $G$ be a nice graph.
Then, for every set $W$ of $|E(G)|+2\Delta(G)$ distinct strictly positive weights,
there exists an edge-injective neighbour-sum-distinguishing $W$-edge-weighting of $G$.
In particular, we have $\chi_{\Sigma}^{e,1}(G) \leq |E(G)|+2\Delta(G)$.
\end{theorem}

\begin{proof}
We may assume that $G$ is connected.
Set $\Delta:=\Delta(G)$, and let $n:=|V(G)|$ and $m:=|E(G)|$ denote the order and size, respectively, of $G$.
Also, set $W:=\{\alpha_1,...,\alpha_{m+2\Delta}\}$ where $\alpha_1 < ... < \alpha_{m+2\Delta}$.
First choose a vertex $v^*$ with degree~$\Delta$ in $G$,
and let $T$ be a spanning tree of $G$ including all edges incident to $v^*$.
From $T$, we deduce a partition $V_0 \cup  ...  \cup V_k$ of $V(G)$,
where each part $V_i$ includes the vertices of $G$ being at distance~$i$ from $v^*$ in $T$.
In particular, $V_0=\{v^*\}$,
and, for every vertex $u$ in a part $V_i$ with $i \neq 0$,
there is exactly one edge from $u$ to $V_{i-1}$ in $T$.
We call this edge the \textit{private edge} of $u$.

We now describe how to obtain an edge-injective neighbour-sum-distinguishing $W$-edge-weighting of $G$.
We start by assigning the edge weights $\alpha_1, ... ,\alpha_{m-(n-1)}$ to the edges of $E(G) \setminus E(T)$ in an arbitrary way.
This leaves us with all edges of $T$ to be weighted, 
which includes at least one incident (private) edge for every vertex different from $v^*$,
and all edges incident to $v^*$.
To weight these edges without creating any conflict,
we will first consider all vertices of $V_k$ and weight their private edges carefully,
then do the same for all vertices of $V_{k-1}$,
and so on layer by layer until all edges of $T$ are weighted.
Fixing any ordering over the vertices of $V_k, ... ,V_1$,
this weighting scheme yields an ordering $u_1, ... ,u_{n-1}$ in which the vertices are considered 
(i.e. the $|V_k|$ first $u_i$'s belong to $V_k$,
the $|V_{k-1}|$ next $u_i$'s belong to $V_{k-1}$, and so on; the $|V_1|$ last $u_i$'s belong to $V_1$).
We note that the private edges of the $|V_1|$ last $u_i$'s go to $v^*$.

To extend the edge-injective neighbour-sum-distinguishing edge-weighting to the edges of $T$ correctly,
we consider the $u_i$'s in order,
and for each of these vertices, we weight its private edge in such a way that no sum conflict arises.
Assume we are currently dealing with vertex $u_i$, meaning that all previous $u_i$'s have been correctly treated.
If $u_i \not \in V_1$, then we assign to the private edge of $u_i$ a non-used weight among $\{\alpha_{m-(n-1)+1}, ... ,\alpha_m\}$ in such a way that $\sigma(u_i)$ gets different from the sums of the at most $\Delta-1$ already treated neighbours of $u_i$.
Note that, even for the last $u_i$ not in $V_1$ to be considered,
the number of remaining non-used weights in $\{\alpha_{m-(n-1)+1}, ... ,\alpha_m\}$ is at least $\Delta+1$,
so this weighting extension can be applied to every vertex.

Now, if $u_i \in V_1$, then we apply the same strategy but with the weights among $\{\alpha_{m+1}, ... ,\alpha_{m+2\Delta}\}$.
Again, even for $u_{n-1}$, note that this set includes at least $\Delta+1$ non-used weights,
so we can correctly choose a weight for $u_{n-1}v^*$ so that $\sigma(u_{n-1})$ gets different from the sums of the previously-treated vertices.
To finish off the proof, we note that, by that strategy,
all edges incident to $v^*$ have been weighted with weights among $\{\alpha_{m+1}, ... ,\alpha_{m+2\Delta}\}$.
Since $d(v^*)=\Delta$, by Observation~\ref{observation:situations} we get that $\sigma(v^*)$ is eventually strictly bigger than the sums incident to its neighbours.
\end{proof}

\medskip

As a concluding remark,
we would like to point out that the $2 \cdot |E(G)|$ bound from Theorem~\ref{theorem:2m},
can, in several situations, be better than the $|E(G)|+2\Delta(G)$ bound from Theorem~\ref{theorem:m+2D}.
To be convinced of that statement, consider the class of graphs obtained by starting from any star with $\Delta$ leaves $u_1, ... ,u_\Delta$
and adding no more than $\Delta-1$ edges joining pairs of vertices among $\{u_1, ... ,u_\Delta\}$.


\section{Refined bounds for particular classes of sparse graphs} \label{section:bound-classes}

We now improve the bounds in Section~\ref{section:upper-bound} to bounds of the form $|E(G)|+k$,
where $k$ is a small constant,
for several classes of nice graphs $G$.
Our weighting strategy here relies on removing some edges from~$G$,
then deducing a correct edge-weighting of the remaining graph,
and extending that weighting to $G$.
So that this weighting strategy applies,
we focus on rather sparse graph classes with particular properties inherited by their subgraphs.
In that respect, we give a special focus to nice $2$-degenerate graphs, and nice graphs with maximum average degree at most~$3$.
It is worth recalling that these graphs may have arbitrarily large maximum degree,
so Theorem~\ref{theorem:m+2D} does not provide the kind of bound we are here interested in.

Throughout this section,
when speaking of a \textit{$k$-vertex},
we mean a degree-$k$ vertex.
By a \textit{$k^-$-vertex} (resp. \textit{$k^+$-vertex}),
we refer to a vertex with degree at most (resp. at least) $k$.


\subsection{$2$-degenerate graphs}

A graph $G$ is said to be \textit{$k$-degenerate} if every subgraph of $G$ has a $k^-$-vertex.
In the next result, we focus on nice $2$-degenerate graphs, 
and exhibit an upper bound on their value of $\chi^{e,1}_\Sigma$.

\begin{theorem} \label{theorem:2deg}
Let $G$ be a nice $2$-degenerate graph.
Then, for every set $W$ of $|E(G)|+4$ distinct strictly positive weights,
there exists an edge-injective neighbour-sum-distinguishing $W$-edge-weighting of $G$.
In particular, we have $\chi_{\Sigma}^{e,1}(G) \leq |E(G)|+4$.
\end{theorem}

\begin{proof}
Assume the claim is false, and let $G$ be a counterexample that is minimal in terms of $n_G+m_G$,
where $n_G:=|V(G)|$ and $m_G:=|E(G)|$.
Set $W:=\{\alpha_1,...,\alpha_{m_G+4}\}$.
We show below that $G$ cannot be a counterexample, and thereby get a contradiction. 
This is done by showing that we can always remove some edges from~$G$
while keeping the graph nice,
then deduce an edge-injective neighbour-sum-distinguishing $\{\alpha_1,...,\alpha_{m_{G'}+4}\}$-edge-weighting $w_{G'}$ of the remaining graph $G'$,
where $m_{G'}:=|E(G')|$,
and finally extend $w_{G'}$ to get an edge-injective neighbour-sum-distinguishing $W$-edge-weighting $w_G$ of $G$.

We start by pointing out properties of $G$ we may assume.
Clearly, we may suppose that $G$ is connected.
According to Observation~\ref{observation:small-degree}, we may also assume that $\Delta(G) \geq 3$,
and, therefore, that $m_G\geq 4$, as otherwise $G$ would be a tree, 
in which case a weighting exists according to Theorem~\ref{theorem:trees}.
We note as well that the $1$-vertices of $G$ must be adjacent to vertices with sufficiently large degree.


\begin{claim} \label{claim:2deg-1vertex}
Every $1$-vertex of $G$ is adjacent to a $6^+$-vertex.
\end{claim}

\begin{proof}
Assume for contradiction that $G$ has a $1$-vertex $u$ adjacent to a $5^-$-vertex $v$.
Let $G':=G-\{uv\}$. 
Then $G'$ is $2$-degenerate, and nice as otherwise $G$ would be a path of length $2$ (in which case Theorem~\ref{theorem:trees} applies).
Thus $G'$ admits an edge-injective neighbour-sum-distinguishing $\{\alpha_1,...,\alpha_{m_{G'}+4}\}$-edge-weighting $w_{G'}$,
where $m_{G'}:=m_G-1$.
According to Observation~\ref{observation:remove-pendant},
we can correctly extend $w_{G'}$ to $uv$, hence to $G$, since we have at least five distinct weights available for that.
This is a contradiction.
\end{proof}

\medskip

From Claim~\ref{claim:2deg-1vertex}, we also deduce the following as a corollary.

\begin{claim} \label{claim:edge-nice}
$G-\{uv\}$ is nice for every edge $uv$.
\end{claim}

\begin{proof}
Let $uv$ be an edge of $G$, and set $G':=G-\{uv\}$. If $d_G(u) \geq 3$ and $d_G(v) \geq 3$, then $G'$ is clearly nice. 
Furthermore, if $d_G(u)=1$ or $d_G(v)=1$, then $G'$ is nice by Claim~\ref{claim:2deg-1vertex}.

Now assume that at least one of $u$ and $v$ has degree~$2$ in $G$.
Without loss of generality, assume that $d_G(u)=2$, and let $u'$ be the neighbour of $u$ different from $v$. 
By Claim~\ref{claim:2deg-1vertex} we have $d_G(v) \geq 2$ and $d_G(u') \geq 2$. 
If $d_G(v) \geq 3$, then clearly $G'$ is nice.
So assume $d_G(v)=2$, and let $v'$ be the neighbour of $v$ different from $u$.
Then, again by Claim~\ref{claim:2deg-1vertex}, we have $d_G(v') \geq 2$, and $G'$ is nice.
\end{proof}

\medskip

As a consequence of Claim~\ref{claim:edge-nice} and Observation~\ref{observation:extending-to-an-edge}, 
we immediately get the following.

\begin{claim} \label{claim:edge-sumdegree6}
$G$ has no edge $uv$ with $d_G(u)+d_G(v)\le 6$. 
\end{claim}

\medskip

We are now ready to start off the proof.
Let $S_1$ denote the set of $2^-$-vertices of $G$, and set $G_1:=G-S_1$.
Since $\Delta(G) \geq 3$, graph $G_1$ has vertices.
In particular, since $G_1$ is $2$-degenerate, it has a $2^-$-vertex $v$.
Let us denote as $d^+(v)$ the number of neighbours, in $G$, of $v$ in $S_1$.
Then $d_G(v)=d^+(v)+d_{G_1}(v)$.

First assume that $d^+(v) \geq 3$, and let $v_1,v_2,v_3$ be three neighbours of $v$ in $S_1$. 
We here consider $G':=G-\{vv_1,vv_2,vv_3\}$.
Note that $G'$ has to be nice,
as otherwise $G$ would have an edge violating Claim~\ref{claim:edge-sumdegree6}.
Due to the minimality of $G$, and because $G'$ is a nice $2$-degenerate graph,
there exists an edge-injective neighbour-sum-distinguishing $\{\alpha_1,...,\alpha_{m_{G'}+4}\}$-edge-weighting $w_{G'}$ of $G'$. 
We extend $w_{G'}$ to $vv_1, vv_2, vv_3$, thus to $G$, 
assigning weights among a set of seven weights including those among $\{\alpha_{m_G+2},\alpha_{m_G+3},\alpha_{m_G+4}\}$
in the following way.

We first assign a weight $\beta_1$ from $\{\alpha_{m_G+3},\alpha_{m_G+4}\}$ to the edge $vv_1$
so that we do not create a sum conflict involving $v_1$ and its neighbour different from $v$ (if any),
which is clearly possible with two distinct weights.
Similarly, we then assign a weight $\beta_2$ from $\{\alpha_{m_G+2},\alpha_{m_G+3},\alpha_{m_G+4}\} \setminus \{\beta_1\}$ to $vv_2$
so that we do not create a sum conflict involving $v_2$ and its neighbour different from $v$ (if any).
Note that due to the choice of $\beta_1$ and $\beta_2$,
which are strictly bigger than the weights among $\{\alpha_1, ... ,\alpha_{m_{G'}+4}\}$,
no matter how we extend the weighting to $vv_3$
it cannot occur that $\sigma_{w_G}(v)$ gets equal to the sum of weights incident to a $2^-$-vertex neighbouring $v$.
Hence, when extending $w_{G'}$ to $vv_3$,
we just have to make sure that $\sigma_{w_G}(v_3)$ does not get equal to the sum of weights incident to the neighbour of $v_3$ different from $v$ (if any),
and that $\sigma_{w_G}(v)$ does not get equal to the sums of weights incident to its at most two neighbours in $G_1$.
So there are at most three conflicts to take into account while we have five weights in hand to weight $vv_3$.
Clearly, this is sufficient to extend the weighting.

\medskip 

Assume now that $d^+(v)=2$ and let $v_1,v_2 \in S_1$ denote the two neighbours of $v$ with degree at most~$2$ in $S_1$. 
We here consider $G':=G-\{vv_1,vv_2\}$, which is 2-degenerate, and nice by Claim~\ref{claim:edge-sumdegree6},
and hence admits an edge-injective neighbour-sum-distinguishing $\{\alpha_1,...,\alpha_{m_{G'}+4}\}$-edge-weighting $w_{G'}$,
where $m_{G'}:=m_G-2$.
Recall that $d_G(v_1),d_G(v_2) \leq 2$,
and that $d_G(v) \leq 4$.
According to Observation~\ref{observation:cut2edges},
we can correctly extend $w_{G'}$ to $vv_1$ and $vv_2$ provided we have at least six distinct weights in hand.
Since this is precisely the case here, an extension of $w_{G'}$ to $G$ exists.

\medskip

The last case to consider is when $d^+(v)=1$, which we cannot directly treat using similar arguments as above.
We may however assume that all $2^-$-vertices $v$ of $G_1$ verify $d^+(v)=1$
as otherwise one of the previous situations would apply.
Furthermore, these vertices have degree exactly~$3$, i.e. they each have exactly two neighbours in $G_1$,
as otherwise they would belong to $S_1$.
Now let $S_2$ denote the set of all $2^-$-vertices of $G_1$
and set $G_2:=G-\{S_1,S_2\}$.
We fix a vertex $v^*$ for the rest of the proof, chosen as follows.
If $G_2$ has vertices, then we choose, as $v^*$,
a vertex of $G_2$ verifying $d_{G_2}(v^*) \leq 2$ (which exists, as $G_2$ is $2$-degenerated).
Otherwise, we choose as $v^*$ one vertex verifying $0<d_{G_1}(v^*) \leq 2$.
In the latter case, note that $v^*$ belongs to $S_2$.

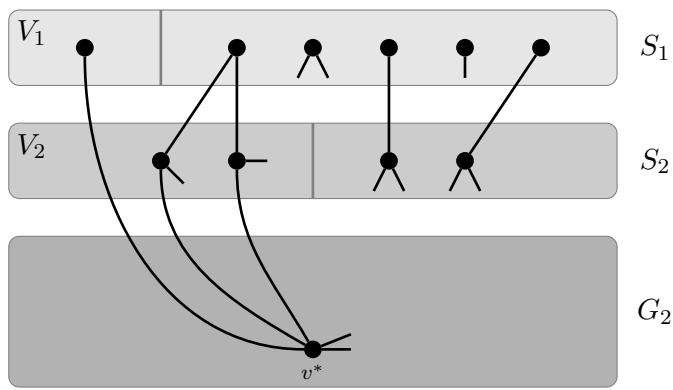
\begin{figure}[t]
	\center 
	\begin{tikzpicture}[inner sep=0.7mm]
		\draw [rounded corners,color=gray,fill=gray!20] (0,0)--(8,0)--(8,-1) -- (0,-1) -- cycle;	
		\draw[-,line width=1pt,gray] (2,0) -- (2,-1);
		\node at (8.5,-0.5) {$S_1$};
		\node at (0.3,-0.3) {$V_1$};
		
		\node[draw,circle,line width=1pt, fill=black] (u1) at (1,-0.5){};	
		\node[draw,circle,line width=1pt, fill=black] (u2) at (3,-0.5){};	
		\node[draw,circle,line width=1pt, fill=black] (u3) at (4,-0.5){};	
		\node[draw,circle,line width=1pt, fill=black] (u4) at (5,-0.5){};	
		\node[draw,circle,line width=1pt, fill=black] (u5) at (6,-0.5){};	
		\node[draw,circle,line width=1pt, fill=black] (u6) at (7,-0.5){};	
		
		\draw [rounded corners,color=gray,fill=gray!40] (0,-1.5)--(8,-1.5)--(8,-2.5) -- (0,-2.5) -- cycle;	
		\draw[-,line width=1pt,gray] (4,-1.5) -- (4,-2.5);
		\node at (8.5,-2) {$S_2$};
		\node at (0.3,-1.8) {$V_2$};
		
		\node[draw,circle,line width=1pt, fill=black] (v1) at (2,-2){};	
		\node[draw,circle,line width=1pt, fill=black] (v2) at (3,-2){};	
		\node[draw,circle,line width=1pt, fill=black] (v3) at (5,-2){};	
		\node[draw,circle,line width=1pt, fill=black] (v4) at (6,-2){};	
		
		\draw [rounded corners,color=gray,fill=gray!60] (0,-3) -- (8,-3) -- (8,-5) -- (0,-5) -- cycle;	
		\node at (8.5,-4) {$G_2$};
		
		\node[draw,circle,line width=1pt, fill=black] (w) at (4,-4.5)[label=below:\scriptsize $v^*$]{};	
		\draw[-,line width=1pt] (w) -- (4.5,-4.5);
		\draw[-,line width=1pt] (w) -- (4.5,-4.3);
		\draw[-,line width=1pt] (w) to[out=120,in=-90] (v2);
		\draw[-,line width=1pt] (w) to[out=150,in=-90] (v1);
		\draw[-,line width=1pt] (w) to[out=180,in=-90] (u1);
		
		\draw[-,line width=1pt] (v1) -- (2.3,-2.3);
		\draw[-,line width=1pt] (v1) -- (u2);
		\draw[-,line width=1pt] (v2) -- (3.4,-2);
		\draw[-,line width=1pt] (v2) -- (u2);
		\draw[-,line width=1pt] (v3) -- (u4);
		\draw[-,line width=1pt] (v3) -- (4.8,-2.4);
		\draw[-,line width=1pt] (v3) -- (5.2,-2.4);
		\draw[-,line width=1pt] (v4) -- (u6);
		\draw[-,line width=1pt] (v4) -- (5.8,-2.4);
		\draw[-,line width=1pt] (v4) -- (6.2,-2.4);
		
		\draw[-,line width=1pt] (u3) -- (3.8,-0.9);
		\draw[-,line width=1pt] (u3) -- (4.2,-0.9);
		\draw[-,line width=1pt] (u5) -- (6,-0.9);
	\end{tikzpicture}

\caption{Illustration of the sets $V_1$ and $V_2$ introduced in the proof of Theorem~\ref{theorem:2deg}.}
\label{figure:deg}
\end{figure}

Now, consider the following sets (see Figure~\ref{figure:deg} for an illustration)
\begin{center}
$V_1 :=\left\{v \in V(G) ~|~v \in S_1 \cap N_G(v^*)\right\}$ \hspace{10pt} and \hspace{10pt} $V_2=\{v \in V(G)~|~v \in S_2 \cap N_G(v^*)\}$,
\end{center}

\noindent and set $d_1^+:=|V_1|$ and $d_2^+:=|V_2|$.
Due to our choice of $v^*$,
we have $d_2^+ \geq 1$.
Furthermore, all vertices in $V_2$ are $3$-vertices adjacent to $v^*$ and to a $2^-$-vertex in $S_1$,
while all vertices in $V_1$ are $2^-$-vertices adjacent to $v^*$.
Also, we have $ d_1^++d_2^+ \leq d_G(v^*) \leq d_1^++d_2^++2$.

First assume that $d_1^++d_2^+ \geq 4$, and let $v_1,v_2,v_3,v_4$ be any four distinct neighbours of $v^*$ in $V_1 \cup V_2$. 
We here set $G'=G-\{v^*v_1,v^*v_2,v^*v_3,v^*v_4\}$.
Since the $v_i$'s are $3^-$-vertices in $G$,
it should be clear, according to Claim~\ref{claim:edge-sumdegree6},
that $G'$ is nice.
As it is also $2$-degenerated, by minimality of $G$
there exists an edge-injective neighbour-sum-distinguishing $\{\alpha_1,...,\alpha_{m_{G'}+4}\}$-edge-weighting $w_{G'}$,
where $m_{G'}:=m_G-4$.

We now have to prove that we can extend $w_{G'}$ to $w_G$
using at most eight distinct weights
including those among $\{\alpha_{m_G+1},\alpha_{m_G+2},\alpha_{m_G+3},\alpha_{m_G+4}\}$.
Since $d^+_2 \geq 1$, some of the $v_i$'s belong to $V_2$;
assume $v_1$ is one such vertex.	
We first assign a weight $\beta_1$ to $v^*v_1$ from the set $\{\alpha_{m_G+2},\alpha_{m_G+3},\alpha_{m_G+4}\}$
so that no conflict involving $v_1$ and one of its two neighbours different from $v^*$ arises.
This is clearly possible with at least three distinct weights.
Similarly, we assign two weights $\beta_2$ and $\beta_3$ from the set $\{\alpha_{m_G+1},\alpha_{m_G+2},\alpha_{m_G+3},\alpha_{m_G+4}\} \setminus \{\beta_1\}$ to $v^*v_2$ and $v^*v_3$, respectively, 
so that no conflict involving $v_2$ or $v_3$ and one of their at most two neighbours different from $v ^*$ arises.
We note that this is possible since, though $v_2$ and $v_3$ might be $3$-vertices,
they are adjacent to a $2$-vertex in that case.
Under the assumption that we assign a weight among $\{\alpha_{m_G+1},\alpha_{m_G+2},\alpha_{m_G+3},\alpha_{m_G+4}\}$ to $v^*v_2$ and $v^*v_3$,
we cannot create any sum conflict involving $v_2$ or $v_3$ and a neighbouring $2$-vertex.
In other words, only one conflict involving $v_2$ or $v_3$ may arise here.

We finally have to extend $w_{G'}$ to $v^*v_4$.
Note that due to the choice of $\beta_1,\beta_2,\beta_3$,
and because $v_4$ is a $3^-$-vertex in $G$,
it cannot be that, currently, the sum of weights incident to $v^*$ is exactly the sum of weighs incident to $v_4$.
Furthermore, for the same reasons,
no matter how we weight $v^*v_4$
it cannot happen that, eventually, $\sigma_{w_G}(v^*)$ gets equal to the sum of weights incident to any vertex in $V_1 \cup V_2$.
Hence, when weighting $v^*v_4$, we just have to make sure that $\sigma_{w_G}(v^*)$ does not get equal to the sums of weights incident to the at most two other neighbours of $v^*$ (i.e. those not in $V_1 \cup V_2$, unless $G_2$ is empty in which case all neighbours of $v^*$ belong to $V_1 \cup V_2$),
and that $\sigma_{w_G}(v_4)$ does not get equal to the sums of weights incident to the at most two neighbours of $v_4$ different from $v^*$.
Since we have five distinct weights left to weight $v^*v_4$,
necessarily one of these weights respect these conditions.
The claimed extension of $w_{G'}$ hence exists.

\medskip

To complete the proof, we have to consider the cases where $d_1^++d_2^+ \leq 3$.
Denote by $v_1$ one neighbour of $v^*$ in $V_2$,
which exists since $d_2^+ \geq 1$.
Since $v_1$ belongs to $V_2$, we know that $v_1$ is a $3$-vertex adjacent to a $2$-vertex, say $u_1$, in $S_1$.
Set $G':=G-\{v^*v_1,v_1u_1\}$. Again, $G'$ is $2$-degenerate and nice by Claim~\ref{claim:edge-sumdegree6}.
So let $w_{G'}$ be an edge-injective neighbour-sum-distinguishing $\{\alpha_1,...,\alpha_{m_{G'}+4}\}$-edge-weighting of $G'$,
which exists due to the minimality of $G$,
where $m_{G'}:=m_G-2$.
For contradiction, we show that $w_{G'}$ can be extended to $G$ and that we can do it with six distinct weights.

The degree properties here are that $d_G(v^*) \leq 5$, $d_G(v_1)=3$ and $d_G(u_1) = 2$.
It can be observed, under those assumptions, that the quantity
$$\mu:=\left(d_G(v^*)+1\right)+\max\left\{0,d_G(v_1)+d_G(u_1)-d_G(v^*)-1\right\}$$
is bounded above by~$6$.
From Observation~\ref{observation:cut2edges},
we hence know that $w_{G'}$ can be extended to $v^*v_1$ and $v_1u_1$, as claimed.
This completes the proof.
\end{proof}


\subsection{Graphs with maximum average degree at most~$3$}

We recall that, for any given graph $G$,
the \textit{maximum average degree} of $G$, denoted $\mad(G)$, 
is defined as the maximum average degree of a subgraph of $G$. 
That is
\begin{center}
$\mad(G) := \max \left\{\frac{2\cdot |E(H)|}{|V(H)|} : H {\rm ~is~a~non\mymathhyphen empty~subgraph~of~} G \right\}$.
\end{center}

In the next result, we prove an upper bound on $\chi^{e,1}_\Sigma$ for every nice graph with maximum average degree at most~$3$.

\begin{theorem} \label{theorem:mad}
Let $G$ be a nice graph with $\mad(G) \leq 3$.
Then, for every set $W$ of $|E(G)|+6$ distinct strictly positive weights,
there exists an edge-injective neighbour-sum-distinguishing $W$-edge-weighting of $G$.
In particular, we have $\chi_{\Sigma}^{e,1}(G) \leq |E(G)|+6$.
\end{theorem}

\begin{proof}
Assume there exists a counterexample to the claim, 
that is, 
there exists a nice graph $G$ for which we have $\mad(G) \leq 3$ but, for a particular set $W$ including $|E(G)|+6$ weights,
there is no edge-injective neighbour-sum-distinguishing $W$-edge-weighting of $G$.
We consider $G$ minimum in terms of $n_G+m_G$,
where $n_G:=|V(G)|$ and $m_G:=|E(G)|$.
Set $W:=\{\alpha_1,...,\alpha_{m_G+6}\}$, where $\alpha_1 <... <\alpha_{m_G+6}$.
Our ultimate goal in this proof is to show that $G$ cannot exist.
The strategy we employ to this end
is essentially to show that $G$ has a nice subgraph~$H$, with order $n_H$ and size $m_H$, such that 
$H$ has an edge-injective neighbour-sum-distinguishing $\{\alpha_1,...,\alpha_{m_H+6}\}$-edge-weighting $w_H$ that can be extended to an edge-injective neighbour-sum-distinguishing $W$-edge-weighting $w_G$ of $G$,
contradicting the fact that $G$ is a counterexample.
The main tool we want to use, in order to show that $H$ has such an edge-weighting, is Theorem~\ref{theorem:m+2D}.
Since $G$ is a counterexample to the claim, note that Theorem~\ref{theorem:m+2D} already implies that $\Delta(G) \geq 4$.
Furthermore, we may assume that $G$ is connected,
and is not a tree as otherwise Theorem~\ref{theorem:trees} would apply.

\medskip

The subgraph $H$ we consider is obtained by removing all $1$-vertices from~$G$.
Of course, we have $\mad(H) \leq 3$ and it may happen that $G=H$.
We may as well assume that $H$ remains nice,
as, if it is not the case,
then $G$ would be a tree (a bistar, i.e. a tree having exactly two $2^+$-nodes, being adjacent),
which is not possible as pointed out above.

In the following result, we observe that, by showing that $H$ verifies $\Delta(H) \leq 3$, then we will get our conclusion.

\begin{proposition} \label{proposition:DeltaH>4}
If $\Delta(H) \leq 3$, then $G$ is not a counterexample.
\end{proposition}

\begin{proof}
If $G=H$, then $G$ admits an edge-injective neighbour-sum-distinguishing $W$-edge-weighting
according to Theorem~\ref{theorem:m+2D} since we would have $\Delta(G) \leq 3$.
So assume that $G$ has $1$-vertices.
Since we assume that $\Delta(H) \leq 3$, 
there exists an edge-injective neighbour-sum-distinguishing $\{\alpha_1,...,\alpha_{m_H+6}\}$-edge-weighting $w_H$ of $H$,
still according to Theorem~\ref{theorem:m+2D}.

We now extend $w_H$ to the pendant edges of $G$.
We successively consider every vertex $v$ of $H$ incident to a pendant edge.
We start by assigning an arbitrary non-used weight to every pendant edge incident to $v$,
but one, say $vu$.

We claim that we can find a correct weight for $vu$.
First, we note, according to Observation~\ref{observation:situations},
that only the neighbours of $v$ in $H$
can eventually cause sum conflicts.
Hence, when extending $w_H$ to $vu$,
we just have to make sure, since $vu$ is the last non-weighted pendant edge incident to $v$,
that $\sigma(v)$ does not meet any of the determined sums of the vertices adjacent to $v$ in $H$.
By our assumption on $\Delta(H)$, there are at most three such vertices,
while we have at least seven ways to weight~$vu$ (among $\{\alpha_1, ... ,\alpha_{m_G+6}\}$),
each determining a distinct value for $\sigma(v)$.
We can hence find a correct non-used weight for $vu$.

Since the process above can be applied for all vertices of $H$ incident to a pendant edge in $G$,
weighting $w_H$ can hence be extended to all pendant edges of $G$.
Thus $w_H$ can be extended to an edge-injective neighbour-sum-distinguishing $W$-edge-weighting of $G$, as claimed.
\end{proof}

\medskip

It remains to show that $\Delta(H) \leq 3$.
This is proved by getting successive information concerning the structure of $H$
so that classical discharging arguments can eventually be employed.

\begin{claim} \label{claim:pendant-deg7}
If $v \in V(H)$ is adjacent to $1$-vertices in $G$, then $d_H(v) \geq 7$.
\end{claim}

\begin{proof}
This follows from Observation~\ref{observation:remove-pendant},
as, when removing a pendant edge from $G$, 
applying induction,
and putting the edge back,
we then have seven distinct weights to achieve the extension to $G$.
\end{proof}

\begin{claim} \label{claim:claim2}
We have $\delta(H) \geq 2$.
\end{claim}

\begin{proof}
If $\delta(H)=0$, then $G$ is a star, contradicting one of our initial assumptions.
Now, if $\delta(H)=1$, then $G$ includes a vertex $v$ such that $d_H(v)=1$ and $v$ is incident to pendant edges in $G$.
But this is impossible as such a $v$ would not meet the condition in Claim~\ref{claim:pendant-deg7}.
So $\delta(H) \geq 2$.
\end{proof}

\begin{claim} \label{claim:no2v2v}
Graph $H$ has no two adjacent $2$-vertices.
\end{claim}

\begin{proof}
Suppose that $H$ has an edge $uv$ such that $d_H(u)=d_H(v)=2$.
Recall that, according to Claim~\ref{claim:pendant-deg7}, we have $d_G(u)=d_G(v)=2$.
In this case, we consider the graph $G':=G-\{uv\}$ with size $m_{G'}:=|E(G')|$.
Clearly $G'$ remains nice (otherwise Claim~\ref{claim:pendant-deg7} would be violated),
has $\mad(G') \leq 3$,
and, due to the minimality of $G$,
graph $G'$ admits an edge-injective neighbour-sum-distinguishing $\{\alpha_1,...,\alpha_{m_{G'}+6}\}$-edge-weighing $w_{G'}$.

In $G'$, we have $d_{G'}(u)=d_{G'}(v)=1$.
Let $u'$ and $v'$ be the neighbours of $u$ and $v$, respectively, in $G'$.
Since $w_{G'}$ is edge-injective, we have $w_{G'}(uu') \neq w_{G'}(vv')$.
We now note that, under all those assumptions,
weighting $w_{G'}$ can easily be extended to an edge-injective neighbour-sum-distinguishing $W$-edge-weighing $w_G$ of $G$,
i.e. to the edge $uv$, a contradiction.
We note that, because $w_{G'}(uu') \neq w_{G'}(vv')$ and $d_G(u)=d_G(v)=2$,
we cannot get $\sigma_{w_G}(u)=\sigma_{w_G}(v)$ when assigning any weight to $uv$, recall Observation~\ref{observation:situations}.
So the only constraints we have
are that $\sigma_{w_G}(u)$ has to be different from $\sigma_{w_G}(u')$ (which is exactly $\sigma_{w_{G'}}(u')$) and $\sigma_{w_G}(v)$ must be different from $\sigma_{w_G}(v')$ (which is exactly $\sigma_{w_{G'}}(v')$).
These constraints forbid us from assigning, to $uv$, at most two of the seven weights that have not been used yet.
So we can extend $w_{G'}$ to $w_G$.
\end{proof}

\begin{claim} \label{claim:no2vonly3v}
Graph $H$ has no $2$-vertex adjacent to two $3$-vertices.
\end{claim}

\begin{proof}
Assume $H$ has such a vertex $v$ with $d_H(v)=2$,
and $v$ has two neighbours $u_1$ and $u_2$ verifying $d_H(u_1)=d_H(u_2)=3$.
According to Claim~\ref{claim:pendant-deg7},
we have $d_G(v)=2$, $d_G(u_1)=3$ and $d_G(u_2)=3$.
Let $G':=G-\{vu_1,vu_2\}$ and $m_{G'}:=|E(G')|$.
Clearly, $G'$ remains nice with $\mad(G') \leq 3$,
and, by the minimality of $G$,
there exists an edge-injective neighbour-sum-distinguishing $\{\alpha_1,...,\alpha_{m_{G'}+6}\}$-edge-weighing $w_{G'}$.
According to Observation~\ref{observation:cut2edges}, weighting $w_{G'}$ can be extended to an edge-injective neighbour-sum-distinguishing $W$-edge-weighing of $G$
provided we have at least five distinct edge weights in hand.
Since we here have eight non-used edge weights dedicated to weighting $vu_1$ and $vu_2$,
the extension of $w_{G'}$ to $G$ hence exists.
\end{proof}

\begin{claim} \label{claim:no3vseestwo2v}
Graph $H$ has no $3$-vertex adjacent to two $3^-$-vertices.
\end{claim}

\begin{proof}
The proof is similar to that of the previous claim.
Assume $H$ has such a $3$-vertex $v$
being adjacent to at least two $3^-$-vertices $u_1$ and $u_2$.
Again, we set $G':=G-\{vu_1,vu_2\}$, and let $w_{G'}$ be an edge-injective neighbour-sum-distinguishing $\{\alpha_1,...,\alpha_{m_{G'}+6}\}$-edge-weighing of $G'$,
where $m_{G'}:=|E(G')|$.
Still according to Observation~\ref{observation:cut2edges},
we know that an extension exists provided we have at least six weights available.
So $w_{G'}$ can correctly be extended to $vu_1$ and $vu_2$,
as eight edge weights can be used in the present context.
\end{proof}

\medskip

Before getting our conclusion, 
we prove two last claims which are a bit more general than what we actually need.	

\begin{claim} \label{claim:6v}
Graph $H$ has no $6$-vertex adjacent to two $2$-vertices.
\end{claim}

\begin{proof}
Assume $H$ has such a $6$-vertex $v$,
and let $u_1$ and $u_2$ denote any two of its neighbouring $2$-vertices.
Recall that $d_H(v)=d_G(v)$, $d_H(u_1)=d_G(u_1)$ and $d_H(u_2)=d_G(u_2)$ according to Claim~\ref{claim:pendant-deg7}.
Let $G':=G-\{vu_1,vu_2\}$ and set $n_{G'}:=|V(G')|$ and $m_{G'}:=|E(G')|$.
Clearly $G'$ is nice (Claims~\ref{claim:pendant-deg7} and~\ref{claim:no2v2v}) with $\mad(G') \leq 3$, and, since $n_{G'}+m_{G'}<n_G+m_G$,
there exists an edge-injective neighbour-sum-distinguishing $\{\alpha_1,...,\alpha_{m_{G'}+6}\}$-edge-weighing $w_{G'}$ of $G'$.
Again according to Observation~\ref{observation:cut2edges},
under these conditions, we know that $w_{G'}$ can be extended to $vu_1$ and $vu_2$ provided we have at least eight weights available.
Since this is precisely the case, we are done.
\end{proof}

\begin{claim} \label{claim:4v5v}
Graph $H$ has no $4$- or $5$-vertex adjacent to at least two $3^-$-vertices.
\end{claim}

\begin{proof}
The proof is very similar to that of Claim~\ref{claim:6v},
and can be mimicked by letting $u_1$ and $u_2$ be two $3^-$-vertices adjacent to $v$.
We then get the same conclusion from Observation~\ref{observation:cut2edges}.
\end{proof}

\medskip

We are now ready to prove that $H$ has maximum degree~$3$.

\begin{claim}
We have $\Delta(H) \leq 3$.
\end{claim}

\begin{proof}
Assume the contrary, namely that $\Delta(H) \geq 4$.
We prove the claim by means of the so-called discharging method,
through a discharging procedure, based on the following rules.

\medskip

To every vertex $v$ of $H$, we assign an initial charge $\omega(v)$ being $d_H(v)-3$.
Since $\mad(H) \leq 3$, we have $$\sum_{v \in V(H)} d_H(v) \leq 3 \cdot n_H,$$
which implies that $$\sum_{v \in V(H)} \omega(v) \leq 0.$$
Without creating or deleting any amount of charge assigned to the vertices,
we now transfer a part of the assigned charges from neighbours to neighbours,
through three discharging rules applied in two successive steps.

In the sequel, by a \textit{weak $3$-vertex} of $H$
we refer to a $3$-vertex neighbouring a $2$-vertex (recall that a $3$-vertex of $H$ is adjacent to at most one $2$-vertex according to Claim~\ref{claim:no3vseestwo2v}).
The first discharging step consists in applying the following rule:

\begin{enumerate}
	\item[(R1)] Every $4^+$-vertex transfers $\frac{1}{4}$ to every adjacent weak $3$-vertex.
\end{enumerate}

\noindent Once the first discharging step has been performed, we then apply the second step,
which consists in applying the following two discharging rules:

\begin{enumerate}
	\item[(R2)] Every weak $3$-vertex transfers $\frac{1}{2}$ to every adjacent $2$-vertex.
	
	\item[(R3)] Every $4^+$-vertex transfers $\frac{1}{2}$ to every adjacent $2$-vertex.
\end{enumerate}

We now compute the final charge $\omega^*(v)$ that every vertex $v$ of $H$ gets
once the two steps above have been performed.
Recall that $\delta(H) \geq 2$ according to Claim~\ref{claim:claim2}.

\begin{enumerate}
	\item If $v$ is a $2$-vertex, then $v$ is adjacent to a $4^+$-vertex,
	and either a weak $3$-vertex or a $4^+$-vertex according to Claims~\ref{claim:no2v2v} and~\ref{claim:no2vonly3v}.
	Through Rules (R2) and (R3), the two neighbours of $v$ both transfer $\frac{1}{2}$ to $v$.
	Hence, $\omega^*(v)=\omega(v)+2\times\frac{1}{2}=0$.
	
	\item If $v$ is a $3$-vertex, then $v$ is either weak, or not.
	If $v$ is not weak, it is not concerned by any of Rules (R1), (R2) and (R3),
	so $\omega^*(v)=\omega(v)=0$.
	Now assume $v$ is a weak $3$-vertex.
	According to Claim~\ref{claim:no3vseestwo2v}, vertex $v$ is adjacent to a $2$-vertex~$u$,
	and two $4^+$-vertices $z_1$ and $z_2$.
	Through Rule (R1), vertex $v$ receives $\frac{1}{4}$ from each of $z_1$ and $z_2$, while, through Rule (R2),
	vertex $v$ then transfers $\frac{1}{2}$ to $u$.
	Therefore, $\omega^*(v)=\omega(v)+2\times\frac{1}{4}-\frac{1}{2}=0$.
	
	\item If $v$ is a $4$- or $5$-vertex, then $v$ is adjacent to at most one vertex being either a $2$-vertex or weak $3$-vertex $u$
	according to Claim~\ref{claim:4v5v}.
	The case where $\omega^*(v)$ is minimum
	is when $v$ is a $4$-vertex and $u$ is a $2$-vertex,
	in which case $v$ transfers $\frac{1}{2}$ to $u$.
	In that case, through Rule~(R3), we get $\omega^*(v)=\omega(v)-\frac{1}{2}=\frac{1}{2}$.
	So, whenever $v$ is a $4$- or $5$-vertex, we get $\omega^*(v) > 0$.
	
	\item If $v$ is a $6$-vertex, then $v$ is adjacent to at most one $2$-vertex according to Claim~\ref{claim:6v}.
	The case where $\omega^*(v)$ gets minimum
	is essentially when $v$ neighbours one $2$-vertex and five weak $3$-vertices.
	In that case, following Rules~(R1) and~(R3), we get $\omega^*(v)=\omega(v)-5\times\frac{1}{4}-\frac{1}{2}=\frac{5}{4}$.
	Hence, we always get $\omega^*(v)>0$ in that case.
	
	\item If $v$ is a $7^+$-vertex, then $v$ transfers most charge when $v$ is adjacent to $d_H(v)$ $2$-vertices.
	In that case, following Rule~(R3) we deduce that $\omega^*(v)=\omega(v)-d_H(v)\times\frac{1}{2}$.
	Under the assumption that $d_H(v)\geq 7$, observe that $\omega(v)>d_H(v)\times\frac{1}{2}$.
	So, again, we always have $\omega^*(v)>0$ in this case.
\end{enumerate}

From the analysis above, we get, because $\Delta(H) \geq 4$,
that $$\sum_{v \in V(H)} \omega(v) \leq 0 < \sum_{v \in V(H)} \omega^*(v),$$
which is impossible as we did not create any new amount of charge when applying the discharging procedure.
Hence, we have $\Delta(H) \leq 3$.
\end{proof}

\medskip

The result now follows from Proposition~\ref{proposition:DeltaH>4}.
\end{proof}

\medskip

Theorem~\ref{theorem:mad} applies to all nice graphs with maximum average degree at most~$3$.
Among the classes of such graphs, we would like to highlight the class of nice planar graphs with girth at least~$6$,
where the \textit{girth} $g(G)$ of a graph $G$ is the length of its smallest cycles.
We refer the reader to e.g.~\cite{BKNRS99},
wherein the authors noticed that, for every planar graph $G$,
we have $$\mad(G) < \frac{2g(G)}{g(G)-2}.$$
This gives that every planar graph $G$ with $g(G) \geq 6$ has $\mad(G) \leq 3$.

\begin{corollary} \label{corollary:planar}
Let $G$ be a nice planar graph $G$ with $g(G) \geq 6$.
Then, for every set $W$ of $|E(G)|+6$ distinct weights,
there exists an edge-injective neighbour-sum-distinguishing $W$-edge-weighting of $G$.
In particular, we have $\chi_{\Sigma}^{e,1}(G) \leq |E(G)|+6$.
\end{corollary}


\section{Discussion} \label{section:ccl}

In this work, we have introduced and studied Conjecture~\ref{conjecture:main}
which stands, in some sense, as a combination of the 1-2-3 Conjecture and the Antimagic Labelling Conjecture.
In particular, as a support to Conjecture~\ref{conjecture:main},
we have pointed out that some families of nice graphs agree with it,
or sometimes almost agree with it, i.e. up to an additive constant term.
Although these results can be regarded as a first step towards Conjecture~\ref{conjecture:main},
it is worth emphasizing that our work does not bring anything new towards attacking the 1-2-3 Conjecture and the Antimagic Labelling Conjecture
but rather concerns some side aspects of these two conjectures.

\medskip

As further work towards Conjecture~\ref{conjecture:main},
it would be interesting exhibiting, for all nice graphs $G$, bounds on $\chi_\Sigma^{e,1}(G)$ of the form $|E(G)|+k$ for a fixed constant~$k$.
One could as well try to get a better bound of the form $k \cdot |E(G)|$
for some $k$ in between~$1$ and~$2$.
Obtaining one such of these two bounds
would already improve the ones we have exhibited in Section~\ref{section:upper-bound}.
It is worth mentioning that our bounds in that section
can slightly be improved by making some choices in a more clever way.
But these improvements would allow us to save a small constant number of weights only,
which is far from the desired improvement we have mentioned earlier.

As another direction, we would also be interested in knowing other classes of nice graphs agreeing with Conjecture~\ref{conjecture:main}
and being not known to be antimagic yet.
Among such classes, let us mention the case of nice bipartite graphs $G$,
for which we did not manage to come up with an $|E(G)|+k$ bound on $\chi_\Sigma^{e,1}(G)$,
for any constant~$k$.
Another such class that would be interesting investigating
is the one of nice subcubic graphs.
We already know that cubic graphs agree with Conjecture~\ref{conjecture:main},
recall Theorem~\ref{theorem:antimagic-classes}.
Furthermore, we also know that nice subcubic graphs $G$, in general, verify $\chi_\Sigma^{e,1}(G) \leq |E(G)|+6$,
recall Theorems~\ref{theorem:m+2D} and~\ref{theorem:mad}.
It nevertheless does not seem obvious how these results can be used
in order to show that nice subcubic graphs agree with Conjecture~\ref{conjecture:main}.
Such a result, though, would be one natural step following Observation~\ref{observation:small-degree}.
Nice planar graphs would also be interesting candidates to investigate,
as we have been mostly successful with sparse classes of nice graphs.
Our result in Corollary~\ref{corollary:planar} may be regarded as a first step towards that direction.

Our results in this paper may also be subject to further investigations.
In particular, there is still a gap for nice $2$-degenerate graphs and graphs with maximum average degree at most~$3$ 
between our bounds in Section~\ref{section:bound-classes} and the bound in Conjecture~\ref{conjecture:main}.
One could as well wonder how to generalize our results to nice $k$-degenerate graphs and graphs with maximum average degree at most~$k$
for larger fixed values of $k$.
In particular, it could be interesting to exhibit, for these graphs $G$,
a general upper bound on $\chi_\Sigma^{e,1}(G)$ of the form $|E(G)|+\mathcal{O}(k)$
involving a small function of $k$.


\section*{Acknowledgements}

This work was initiated when the second author visited the Technical University of Denmark.
Support from the ERC Advanced Grant GRACOL, project no. 320812, is gratefully acknowledged.

\end{document}